\documentclass[english, 11pt]{amsart}

\usepackage{tikz}
\usepackage{aliascnt}
\usepackage{mathrsfs}
\usepackage[all,poly,knot]{xy}
\usepackage{hyperref}
\usepackage{comment}  
\usepackage{csquotes}   
\usepackage{amssymb,amsbsy,amsmath,amsfonts,amssymb,amscd,
	graphics,color,footmisc,fancyhdr,multicol,fancybox,
	graphicx,mathrsfs,rotating,ifthen,wasysym}

\usepackage{times}
\usepackage{pdfpages}
\usepackage{multirow}
\usepackage{nccrules}
\usepackage{textcomp}
\usepackage{comment}
\usepackage{framed}
\usepackage[all]{xy}
\usepackage{graphicx}
\usepackage{pgf,tikz}
\usepackage{mathrsfs}
\usetikzlibrary{arrows}
\usepackage{lipsum}
\usepackage{mathtools}

\DeclareMathOperator{\dif}{\text{\normalfont d}}

\DeclareMathOperator{\FS}{FS}

\DeclareMathOperator{\Gr}{Gr}

\DeclareMathOperator{\supp}{supp}
\DeclareMathOperator{\ord}{ord}

\def\log{\mathrm{log}\,}

\theoremstyle{plain}

\newtheorem{thm}{Theorem}[section]  

\newtheorem{cor}[thm]{{Corollary}} 

\newtheorem{lem}[thm]{{Lemma}}

\newtheorem{pro}[thm]{Proposition}

\newtheorem{rem}[thm]{Remark}

\newtheorem*{claim}{Claim}

\theoremstyle{remark}
\newtheorem{rmk}[thm]{Remark}

\numberwithin{equation}{section}

\usepackage[top=1.5in, bottom=1.5in, left=1in, right=1in]{geometry}

\usepackage[hyperpageref]{backref}

\usepackage{xcolor}
\hypersetup{
	colorlinks,
	linkcolor={red!50!black},
	citecolor={blue!62!black},
	urlcolor={blue!80!black}
}
\theoremstyle{plain}
\newcommand{\thistheoremname}{}
\newtheorem*{genericthm*}{\thistheoremname}
\newenvironment{namedthm*}[1]{\renewcommand{\thistheoremname}{#1}%
	\begin{genericthm*}}
	{\end{genericthm*}}

\newtheoremstyle{named}{}{}{\itshape}{}{\bfseries}{.}{.5em}{\thmnote{#3's }#1}
\theoremstyle{named}

\makeatletter
\newcommand\thankssymb[1]{\textsuperscript{\@fnsymbol{#1}}}
\makeatother
\begin{document} 
	\title[Some variants of the generalized Borel Theorem and applications]{\bf Some variants of the generalized Borel Theorem\\ and applications}

	
	\author{Dinh Tuan Huynh\\
	Hue Geometry--Algebra Group}
	
	\address{Hue Geometry--Algebra Group, Department of Mathematics, University of Education, Hue University, 34 Le Loi St., Hue City, Vietnam}
	\email{dinhtuanhuynh@hueuni.edu.vn}

\begin{abstract} In the first part of this paper, we establish some results around generalized Borel's Theorem. As an application, in the second part, we construct example of smooth surface of degree $d\geq 19$ in $\mathbb{CP}^3$ whose complements is hyperbolically embedded in $\mathbb{CP}^3$. This improves the previous construction of Shirosaki where the degree bound $d=31$ was gave. In the last part, for a Fermat-Waring type hypersurface $D$ in $\mathbb{CP}^n$ defined by the homogeneous polynomial
\[
\sum_{i=1}^m h_i^d,
\]
where  $m,n,d$ are positive integers with $m\geq 3n-1$ and $d\geq m^2-m+1$, where $h_i$ are homogeneous generic linear forms on $\mathbb{C}^{n+1}$, for a {\sl nonconstant} holomorphic function $f\colon\mathbb{C}\rightarrow\mathbb{CP}^n$ whose image is not contained in $\supp D$, we establish a Second Main Theorem type estimate:
\[
\big(d-m(m-1)\big)\,T_f(r)\leq N_f^{[m-1]}(r,D)+S_f(r).
\]
This quantifies the hyperbolicity result due to Shiffman-Zaidenberg and Siu-Yeung.
\end{abstract}
\maketitle

\section{Introduction}
It was conjectured by Kobayashi in 1970 \cite{Kobayashi1970}  that a general hypersurface $D$ in projective space $\mathbb{CP}^n$ of degree $d$ large  enough is hyperbolic. According to Zaidenberg \cite{Zaidenberg1987}, the expected optimal degree bound 
should be $d=2n-1$. In the so-called {\it logarithmic} case, it is also anticipated that if $d\geq 2n+1$, then the complement $\mathbb{CP}^n\setminus D$ is also  hyperbolic. The subject of hyperbolicity has attracted much attention and research, partly because it has been believed to be intimately related to Diophantine Geometry. For instance, Lang conjectured that an algebraic variety $V_{\mathbb{K}}$ defined over a number field $\mathbb{K}$ can contain only finitely many $\mathbb{K}$-rational points
provided that $V_{\mathbb{C}}$ is  hyperbolic after some base change $\mathbb{K}\hookrightarrow\mathbb{C}$.

Many works have been done during recent decades towards the above conjectures. Notably, these  conjecture were confirmed under the condition that the degree of $D$ is very high compared with the dimension. In the case of surfaces in $\mathbb{CP}^3$, by studying the entire leaves of foliations on surfaces, the first proof was given by McQuillan \cite{Mcquillan1999} with degree bound $d\geq 36$. Demailly-El Goul \cite{Demailly_Goul2000} provided a better degree bound $d\geq 21$. P\v{a}un \cite{mihaipaun2008} employed the technique of using slanted vector field of Siu \cite{Siu2004} together with the work of McQuillan \cite{Mcquillan1998} to improve the degree bound to $d\geq 18$. Subsequently, hyperbolicity of generic three-folds in $\mathbb{CP}^4$ was confirmed \cite{Rousseau2007, Diverio-Trapani2010} with degree bound $d\geq 593$.

In the case of arbitrary dimension $n$, by generalizing the variational approach of Clemens \cite{Clemens1986} and Voisin \cite{Voisin1996}, Siu \cite{Siu2004} outlined a strategy which finally led to the proof of Kobayashi's conjecture \cite{Siu2015} with very high degree bound $d(n)\gg 1$. Many works are influenced by Siu's program, especially the significant results of Diverio-Merker-Rousseau \cite{DMR2010} on Green-Griffiths' conjecture (see also \cite{Demailly2012, Lionel2016}). Recently, based on the method of using Wronskian differential operators \cite{Xie2018, Brotbek-Lionel2017}, Brotbek gave a new proof for hyperbolicity of high degree generic hypersurface
\cite{Brotbek2017}. The explicit degree bound was given shortly afterwards \cite{YaDeng2017, Demailly2018}. Very recently, B\'{e}rczi-Karwan \cite{Berczi-Kirwan2024} reached the degree bound of
polynomial growth by an approach from geometric invariant theory.

On the other side, in the past five decades, it has been considered a very challenging
problem to construct new examples of hyperbolic hypersurfaces of low degree in projective spaces, as the motto goes: ``the lower degree, the more difficult the hyperbolicity''. The first example of compact hyperbolic manifold due to Brody-Green \cite{Brody-Green1977} is a surface in $\mathbb{CP}^3$ defined by the equation

\[
z_0^d+z_1^d+z_2^d+z_3^d+ (\epsilon\,z_0z_1)^{\frac{d}{2}}+(\epsilon\,z_0z_2)^{\frac{d}{2}}
=
0,
\]
where $d=2k\geq 50$ and $|\epsilon|$ is small enough. After that several hyperbolic hypersurfaces in low dimension were constructed to improve the degree bound \cite{Nadel1989, Elgoul96, siu_yeung1997, Shirosaki2000-01, Shirosaki2000-02,  Fujimoto2001, shiffman_zaidenberg2002_p3, shiffman_zaidenberg_2000, Shiffman_Zaidenberg2005, Zaidenberg_3fold2003, Duval2004}. Duval \cite{Duval2004} constructed a sextic hyperbolic surface in $\mathbb{CP}^3$, which is the lowest degree found up to date. Some recent constructions could also reach this degree bound \cite{Ciliberto_Zaidenberg2013, Duval2017}. It is still not known yet whether there exists any quintic hyperbolic surface.

Currently, there are two main methods of constructing hyperbolic hypersurfaces in projective space. The first one is to seek them among the class of perturbations of Fermat hypersurface. This method mades use some variants of the generalized Borel Theorem to study the degeneracy of entire holomorphic curves. The second one is the deformation method introduced by Zaidenberg \cite{Zaidenberg1988}, whose main idea is to find hyperbolic hypersurfaces in the linear pencil of hypersurfaces  $\Sigma_{\epsilon}=\{s_0+\epsilon s=0\}$, where $S_0=\{s_0=0\}$ is a singular hypersurface, $S=\{s=0\}$ is a generic hypersurface and $\epsilon$ is small enough.

The first examples in any dimension were given by Masuda-Noguchi \cite{Masuda_Noguchi1996} with very large degree. Explicit constructions with large degree were gave by Fujimoto \cite{Fujimoto2001,Fujimoto2003}. Some constructions with lower degree asymptotic were provided  by Siu-Yeung \cite{siu_yeung1997} with $d(n)\geq 16\,(n-1)^2$ and by Shiffman-Zaidenberg \cite{shiffman_zaidenberg2002_pn} with $d(n)\geq 4\,(n-1)^2$. Currently, the best asymptotic degree bound is $d\geq {\textstyle{\big(\frac{n+2}{2}\big)^2}}$, obtained in \cite{Huynh2016}. In low dimension $3\leq n\leq 6$, example of hyperbolic hypersurface of lowest degree bound $d=2n$ was  given in \cite{Huynh2015}.

On the other hand,  example of smooth algebraic curve in $\mathbb{CP}^2$ of degree $d\geq 5$ whose complement is hyperbolically embedded was constructed by Zaidenberg \cite{Zaidenberg1988} using deformation. However,  in the logarithmic case, from dimension $n\geq 3$, this method is not effective. Examples of hyperbolically embedded hypersurfaces  were provided by Masuda-Noguchi \cite{Masuda_Noguchi1996} in all dimension, and by Shirosaki \cite{Shirosaki2000-01} in low dimension $n\leq 4$ with better degree bound. Generalizing the construction of Shirosaki to arbitrary dimension, Shiffman-Zaidenberg showed the existence of hyperbolically embedded hypersurfaces in $\mathbb{CP}^n$ with all degree $d\geq 4n^2-2n+1$.

Our first aim in this paper is to  improve the previous degree bound $d\geq 31$ for hyperbolically embedded surfaces in $\mathbb{CP}^3$ by Shirosaki \cite{Shirosaki2000-01}. We first prove a Borel type result for holomorphic curves into projective space avoiding or contained in the hypersurface defined by the homogeneous polynomials of the form 
\[
\sum_{i=0}^{n}z_i^{d-\delta_i}Q_i.
\]
Similar results have been given in \cite{Nadel1989}, \cite{Elgoul96} \cite{Tiba2012} using meromorphic connections (see also \cite{Yang2017}, \cite{Thin2017} for another approach via value distribution theory). Following the arguments of Green, we obtain a stronger degeneracy statement (see Section \ref{section: preparation} for details). As an application, we construct a smooth algebraic surface in $\mathbb{CP}^3$ of degree $d$ for any $d\geq 19$ such that its complement is hyperbolically embedded in $\mathbb{CP}^3$. 

\begin{namedthm*}{Theorem A}
	\label{theorem B}
Let $d\geq 19$ be an integer. There exist some nonzero complex numbers $a_0,a_1,a_2,a_3$ such that  the surface $D\subset\mathbb{CP}^3$  defined in the homogeneous coordinates $[z_0:z_1:z_2:z_3]$ of $\mathbb{CP}^3$ by the homogeneous polynomial
\[
z_0^{d}+z_1^{d-2}(z_1^2+a_0z_0^2)+z_2^{d-2}(z_2^2+a_1z_0^2)+z_3^{d-2}(a_2z_1^2+a_3z_2^2+z_3^2)
\]	
is smooth, and  the complement $\mathbb{CP}^3\setminus D$ is  hyperbolically embedded in $\mathbb{CP}^3$.
\end{namedthm*}

On the other hand, in the quantitative aspect of the hyperbolicity problem, called value distribution theory or Nevanlinna theory,  one studies the frequency of impacts of  entire holomorphic curves into projective space and a family of hypersurfaces, by means of certain Second Main Theorem type estimate.

Before entering the details of the next result, we give a brief introduction to Nevanlinna theory in projective space. For a positive number $r>0$, we denote by $\Delta_r\subset \mathbb{C}$ the disk  of radius $r$ centered at the origin. Fix a truncation level $m\in \mathbb{N}\cup \{\infty\}$, 
for an effective divisor $E=\sum_i\alpha_i\,a_i$ on $\mathbb{C}$ where $\alpha_i\geq 0$, $a_i\in\mathbb{C}$,  the $m$-truncated degree of the divisor $E$ on  the disk $\Delta_r$ is given by
\[
n^{[m]}(r,E)
:=
\sum_{a_i\in\Delta_r}
\min
\,
\{m,\alpha_i\},
\]
the \textsl{truncated counting function at level} $m$ of $E$ is then defined by taking the logarithmic average
\[
N^{[m]}(r,E)
\,
:=
\,
\int_1^r \frac{n^{[m]}(t, E)}{t}\,\dif\! t
\eqno
{{\scriptstyle (r\,>\,1)}.}
\]
When $m=\infty$, for abbreviation we  write $n(t,E)$, $N(r,E)$ for $n^{[\infty]}(t,E)$, $N^{[\infty]}(r,E)$.

Let $f\colon\mathbb{C}\rightarrow \mathbb{P}^n(\mathbb{C})$ be an entire holomorphic curve having a reduced representation $f=[f_0:\cdots:f_n]$ in the homogeneous coordinates $[z_0:\cdots:z_n]$ of $\mathbb{P}^n(\mathbb{C})$. Let $D=\{Q=0\}$ be a divisor in $\mathbb{P}^n(\mathbb{C})$ defined by a homogeneous polynomial $Q\in\mathbb{C}[z_0,\dots,z_n]$ of degree $d\geq 1$. If $f(\mathbb{C})\not\subset D$, then $f^*D=\sum_{a\in\mathbb{C}}\ord_af^*Q$ is a divisor on $\mathbb{C}$. We then define the \textsl{truncated counting function} of $f$ with respect to $D$ as
\[
N_f^{[m]}(r,D)
\,
:=
\,
N^{[m]}\big(r,f^*D\big),
\]
which measures the  frequency of the intersection  $f(\mathbb{C})\cap D$. Next,
the \textsl{proximity function} of $f$ associated to the divisor $D$ is given by
\[
m_f(r,D)
\,
:=
\,
\int_0^{2\pi}
\log
\frac{\big\Vert f(re^{i\theta})\big\Vert_{\max}^d\,
	\Vert Q\Vert_{\max}}{\big|Q(f)(re^{i\theta})\big|}
\,
\frac{\dif\!\theta}{2\pi},
\]
where $\Vert Q\Vert_{\max}$ is the maximum  absolute value of the coefficients of $Q$ and where
\begin{equation}
\label{| |max definition}
\big\Vert f(z)\big\Vert_{\max}
:=
\max
\{|f_0(z)|,\dots,|f_n(z)|\}.
\end{equation}
Since $\big|Q(f)\big|\leq
\left(\substack{d+n\\ n}
\right)\,
\Vert Q\Vert_{\max}\cdot\Vert f\Vert_{\max}^d$, we see that $m_f(r,D)\geq O(1)$ is bounded  from below by some constant.
Lastly, the \textsl{Cartan order function} of $f$ is defined by
\[
T_f(r)
:=
\frac{1}{2\pi}\int_0^{2\pi}
\log
\big\Vert f(re^{i\theta})\big\Vert_{\max} \dif\!\theta
=
\int_1^r\dfrac{\dif\! t}{t} \int_{|z|<t}f^*\omega_{FS}
+O(1)
\eqno
\scriptstyle (r\,>\,1),
\]
capturing the growth of the area of the image of the disks under $f$, with respect to the Fubini--Study metric $\omega_{\FS}$.
The Nevanlinna theory is then established by comparing the above three functions. It consists of two fundamental theorems (for  backgrounds and recent progresses in Nevanlinna theory, see Noguchi-Winkelmann \cite{Noguchi-Winkelmann2014} and Ru \cite{Ru2021}). The first one is a reformulation of the Poisson-Jensen formula.

\begin{namedthm*}{First Main Theorem}\label{fmt} Let $f\colon\mathbb{C}\rightarrow \mathbb{P}^n(\mathbb{C})$ be a holomorphic curve and let $D$ be a hypersurface of degree $d$ in $\mathbb{P}^n(\mathbb{C})$ such that $f(\mathbb{C})\not\subset D$. Then one has the estimate
	\[
	m_f(r,D)
	+
	N_f(r,D)
	\,
	=
	\,
	d\,T_f(r)
	+
	O(1)
	\]
	for every $r>1$,
	whence
	\begin{equation}
	\label{-fmt-inequality}
	N_f(r,D)
	\,
	\leq
	\,
	d\,T_f(r)+O(1).
	\end{equation}
\end{namedthm*}

Hence the First Main Theorem provides an upper bound on the counting function in term of the order function. The reverse direction, called {\sl Second Main Theorem}, is usually much harder, and one often needs to take the sum of the counting functions of many divisors.

Throughout this paper, for an entire curve $f$, by  $S_f(r)$, we mean a real function of $r \in \mathbb{R}^+$ such that 
\[
S_f(r) \leq
O(\log(T_f(r)))+ \epsilon \log r
\]
for every positive constant $\epsilon$ and every $r$ outside of a subset (depending on $\epsilon$) of finite Lebesgue measure of $\mathbb{R}^+$. In the case where $f$ is rational, we understand that $S_f(r)=O(1)$. In any case, there holds

\[
\liminf_{r\rightarrow\infty}\dfrac{S_f(r)}{T_f(r)}
=
0.
\]

A family $\{D_i\}_{1\leq i\leq q}$  of $q\geq n+2$ hypersurfaces in $\mathbb{CP}^n$ is said to be {\sl in general position} if $\cap_{i\in I}D_i=\varnothing$ for any subset $I\subset\{1,\dots,q\}$ of cardinality $n+1$. For a linearly nondegenerate entire curve
$f:\mathbb{C}\rightarrow\mathbb{P}^n(\mathbb{C})$ and for a  family of $q\geq n+2$ hyperplanes $\{H_i\}_{i\,=\,1,\dots,\, q}$ in {\sl general position}, Cartan~\cite{Cartan1933} established a second main theorem
\begin{equation}
\label{Cartan SMT}
(q-n-1)
\,
T_f(r)
\,
\leq
\,
\sum_{i=1}^q N_f^{[n]}(r,H_i)+S_f(r),
\end{equation}
which implies the defect relation
\[
\sum_{i=1}^{q}\delta_{f}^{[n]}(H_i)
\leq n+1.
\]
In the particular case $n=1$, Cartan recovered the classical Nevanlinna's Second Main Theorem. In the collinear case, by a purely potential theoretic approach, a Second Main Theorem for non-constant holomorphic curves $f$ in $\mathbb{CP}^n$ and family of  hypersurfaces $\{D_i\}_{i=1}^q$ in general position without truncation  were given by Eremenko-Sodin \cite{Eremenko-Sodin1992} which implies a defect relation bounded by $2n$. Assuming the algebraically nondegenerate condition for entire curves,  Ru  \cite{Minru2004, Minru2009} obtained a defect relation bounded by $n+1$ by method from Diophantine approximation. Second Main Theorem quantifying the logarithmic Green-Griffiths and Kobayashi conjecture were given in \cite{HVX17}, \cite{Brotbek-Deng2019}, using the logarithmic jet differentials.

Currently, from dimension $n\geq 2$, there are only two results about Second Main Theorem for non constant holomorphic curve $f$ into $\mathbb{CP}^n$. These are the Second Main Theorems of Eremenko-Sodin \cite{Eremenko-Sodin1992} and Brotbek-Deng \cite{Brotbek-Deng2019}. In these works, one either needs many targets (at least $2n+1$) or very large degree (exponential growth compared with the dimension). Our next purpose in this paper is to seek Second Main Theorem for {\sl non-constant} entire holomorphic curves into $\mathbb{CP}^n$ and a hypersurface in this space whose complement is hyperbolically embedded in $\mathbb{CP}^n$. We shall deal with the hypersurface of Fermat-Waring type constructed by Siu-Yeung \cite{siu_yeung1997} and Shiffman-Zaidenberg \cite{shiffman_zaidenberg2002_pn}. Here is the statement of our next result.
\begin{namedthm*}{Theorem B}
	Let $m,n,d$ be positive integers with $m\geq 3n-1$ and $d\geq m^2-m+1$. Let $D$ be the Fermat-Waring type hypersurface in $\mathbb{CP}^n$ defined by the homogeneous polynomial
	\[
	\sum_{i=1}^m h_i^d,
	\]
	where $h_i$ are generic homogeneous linear forms on $\mathbb{C}^{n+1}$. For a {\sl nonconstant} holomorphic function $f\colon\mathbb{C}\rightarrow\mathbb{CP}^n$ whose image is not contained in $\supp D$, the following Second Main Theorem type estimate holds
	\[
	\big(d-m(m-1)\big)\,T_f(r)\leq N_f^{[m-1]}(r,D)+S_f(r).
	\]
\end{namedthm*}

Notably, this implies a defect relation involving the degree of the hypersurface $D$. It shows that the defect of $f$ with respect to $D$ is small when the degree of $D$ is large (see \cite{Tiba2012}, \cite{Yang2017}, \cite{Thin2017} for related results).
\begin{namedthm*}{Defect relation}
	Let $m,n,d$ be positive integers with $m\geq 3n-1$ and $d\geq m^2-m+1$. Let $D$ be the Fermat-Waring type hypersurface in $\mathbb{CP}^n$ defined by the homogeneous polynomial
	\[
	\sum_{i=1}^m h_i^d,
	\]
	where $h_i$ are generic homogeneous linear forms on $\mathbb{C}^{n+1}$. For a {\sl nonconstant} holomorphic function $f\colon\mathbb{C}\rightarrow\mathbb{CP}^n$ whose image is not contained in $\supp D$, the following defect inequality holds
	\[
	\delta_{f}^{[m-1]}(D)
	\leq
	\dfrac{m(m-1)}{d}.
	\]
	Consequently
	\[
	\lim_{d\rightarrow\infty}
	\delta_{f}^{[m-1]}(D)
	=
	0.
	\]
\end{namedthm*}
From the above defect relation, and in view of the fundamental conjecture for entire curves of Griffiths, one may expect the following
\begin{namedthm*}{Conjecture}
	Let $f\colon\mathbb{CP}^n$ be a nonconstant entire holomorphic curve. Let $D$ be a generic hypersurface of degree $d>2n$. If the image of $f$ is not contained in $D$, then the following defect relation holds
	\[
	\delta_{f}(D)
	\leq
	\dfrac{2n}{d}.
	\]
\end{namedthm*}

\section*{Acknowledgement}
A part of this article was written while the author was visiting Vietnam Institute for Advanced Study in Mathematics (VIASM).
\section{Some variants of generalized Borel's Theorem}
\label{section: preparation}
As pointed out by Shiffman in a private conversation with Siu  \cite{siu1990}, from the Cartan Second Main Theorem with truncated counting functions, one can deduce a Second Main Theorem for algebraically nondegenerate holomorphic curves and a family of Fermat type hypersurfaces. Using this technique, some generalizations have been done recently for  small perturbations of Fermat type hypersurfaces \cite{Yang2017}, \cite{Thin2017} (see also \cite{Tiba2012} for a geometric approach via connections). For later purpose, we present here a slightly modifications of these results with a simplified proof.
\begin{pro}
	\label{generalized borel}
Let $d,n,\delta_0,\dots,\delta_n$ be integer numbers with $\delta_i\geq 0$, $n\geq 2$ and $d>n(n+1)+\sum_{i=0}^{n}\delta_i$. Let $Q_i$ ($0\leq i\leq n$) be homogeneous polynomials of degree $\delta_i$. Suppose that the family of hypersurfaces $\{D_i\}_{0\leq i\leq n}$ where $D_i=\{z_i^{d-\delta_i}Q_i=0\}$ is in general position in $\mathbb{CP}^n$. Let $D$ be the hypersurface in $\mathbb{CP}^n$ defined by the homogeneous polynomial 
\[
\sum_{i=0}^{n}z_i^{d-\delta_i}Q_i.
\]
 Let $f\colon\mathbb{C}\rightarrow\mathbb{CP}^n$ be a non-constant holomorphic curve. If there exists no nontrivial linear relation among $(z_i^{d-\delta_i}Q_i)\circ f$, then
\[
\big[d-\big(n(n+1)+\sum_{i=0}^{n}\delta_i\big)\big] T_f(r)\leq
N_f^{[n]}(r,D)+S_f(r).
\]
\end{pro}
\begin{proof} Since $\{z_i^{d-\delta_i}Q_i\}_{0\leq i\leq n}$ is in general position, for any $w\in\mathbb{S}_{n+1}:=\{w\in\mathbb{C}^{n+1}\setminus\{0\},\|w\|=1\}$, there exists at least one index $i$ with $(z_i^{d-\delta_i}Q_i)(w)\not=0$. Thus, by compactness property of $\mathbb{S}_{n+1}$, there exist constants $C_1,C_2>0$ such that
\[
C_1\leq \max_{0\leq i\leq n}\big|(z_i^{d-\delta_i}Q_i)(w)\big|^{\frac{1}{d}}
\leq C_2\eqno(\forall\,w\in\,\mathbb{S}_{n+1}).
\]
Since $z_i^{d-\delta_i}Q_i$ are homogeneous of degree $d$, the above estimate implies
\[
C_1\leq \max_{0\leq i\leq n}\dfrac{\big|(z_i^{d-\delta_i}Q_i)(w)\big|^{\frac{1}{d}}}{\|w\|}
\leq C_2\eqno(\forall\,w\in\,\mathbb{CP}^n).
\]
Now  let $f=[f_0:\dots:f_n]$ be a reduced representation of $f$. Setting
\begin{align*}
\pi\colon\mathbb{CP}^n&\rightarrow\mathbb{CP}^n,\qquad [z_0:\dots:z_n]\mapsto [z_0^{d-\delta_0}Q_0:\dots:z_n^{d-\delta_n}Q_n],\\
g=[g_0:\dots:g_n]\colon \mathbb{C}&\rightarrow\mathbb{CP}^n,\qquad z\mapsto
\pi\circ f(z),
\end{align*}
where $g_i:=(z_i^{d-\delta_i}Q_i)\circ f$. It follows from the above estimate that 
\begin{equation}
\label{comparing order functions of f and g, generalized borel}
T_g(r)=d\,T_f(r)+O(1).
\end{equation}
Let $\{H_i\}_{0\leq i\leq n+1}$ be the family of $n+2$ hyperplanes in $\mathbb{CP}^{n}$ given by 
\begin{align*}
H_i&=\{z_i=0\}\qquad(0\,\leq\,i\,\leq\,n),\\
H_{n+1}&=\{\sum_{i=0}^{n}z_i=0\},
\end{align*}
which is in general position. By assumption, the map $g$ is linearly nondegenerate. Hence, applying Cartan's Second Main Theorem for $g$ and $\{H_i\}$, one receives
\begin{equation}
\label{applying cartan to g and Hi, generalized borel}
T_g(r)
\leq
\sum_{i-0}^{n}N_g^{[n]}(r,H_i)
+
N_g^{[n]}(r,H_{n+1}).
\end{equation}

For $0\leq i\leq n$, the curve $g$ intersects  $H_i$ if and only if the curve $f$ intersects $H_i$ or $Q_i$. Furthermore, if $f$ intersects such $H_i$, then $g$ intersects $H_i$ with multiplicity at least $d$. Hence, from these observations, one gets
\[
N_g^{[n]}(r,H_i)
\leq 
N_f(r,Q_i)+\dfrac{n}{d}N_g(r,H_i),\eqno (0\,\leq\, i\,\leq\,n)
\]
which implies 
\[
N_g^{[n]}(r,H_i)
\leq 
\delta_i\,T_f(r)
+
\dfrac{n}{d}T_g(H_i)
,\eqno (0\,\leq\, i\,\leq\,n),
\] 
by the First Main Theorem. Obviously, one also has $N_g^{[n]}(r,H_{n+1})=N_f^{[n]}(r,D)$. Combining these facts together with \eqref{comparing order functions of f and g, generalized borel}, \eqref{applying cartan to g and Hi, generalized borel}, we receive the desired estimate.
\end{proof}
One can deduce immediately from Proposition~\ref{generalized borel} that, under the same assumptions therein, if $f$ avoids $D$, then it must be (algebraically) degenerate. By following the arguments  in \cite[Example 3.10.22]{Kobayashi1998}, one can actually obtain stronger degeneracy result.
\begin{namedthm*}{Generalized Borel's Theorem (logarithmic case)}
Let $d,n,\delta_0,\dots,\delta_n$ be integer numbers with $\delta_i\geq 0$, $n\geq 2$ and $d>n(n+1)+\sum_{i=0}^{n}\delta_i$. Let $Q_i$ ($0\leq i\leq n$) be homogeneous polynomials of degree $\delta_i$. Suppose that the hypersurfaces  $D_i=\{z_i^{d-\delta_i}Q_i=0\}$ ($0\leq i\leq n+1$) have empty intersection. Then, for the collection of $n+1$ entire holomorphic functions $f_i$ ($0\leq i\leq n$) such that
\[
\sum_{i=0}^{n}f_i^{d-\delta_i}Q_i(f_0,f_1,\dots,f_n)
\]
is nowhere vanishing, there is a partition of indexes $\{0,\dots,n\}=\cup_{\alpha=0}^{\ell}I_{\alpha}$ such that the followings hold
\begin{enumerate}
\item[(i)] $f_i^{d-\delta_i}Q_i(f_0,f_1,\dots,f_n)\equiv 0$ if and only if $i\in I_0$ (of course the set $I_0$ may be empty);
\item[(ii)] The cardinality $|I_{\alpha}|\geq 2$ for every $1\leq\alpha\leq \ell$, with at most one exception;
\item[(iii)] For each $1\leq \alpha\leq \ell$, for arbitrary indexes $i,j\in I_{\alpha}$, there exists a constant $c_{ij}\in\mathbb{C}$ such that $\frac{f_i^{d-\delta_i}Q_i(f_0,f_1,\dots,f_n)}{f_j^{d-\delta_j}Q_j(f_0,f_1,\dots,f_n)}=c_{ij}$;
\item[(iv)] $\sum_{i\in I_{\alpha}}f_i^{d-\delta_i}Q_i(f_0,f_1,\dots,f_n)\equiv 0$ for all $0\leq\alpha\leq\ell$ with one exception.
\end{enumerate}
\end{namedthm*}

\begin{proof}
Consider the map $f\colon\mathbb{C}\rightarrow\mathbb{CP}^n$ given as $f=[f_0:\dots:f_n]$. Let the map $g=[g_0:\dots:g_n]$ and the family of hyperplanes $\{H_i\}_{0\leq i\leq n+1}$ be as in the proof of Proposition~\ref{generalized borel}. Let $I_0=\{0\leq i\leq n: g_i\equiv 0\}$ and $J=\{0,\dots,n\}\setminus I_0$. Suppose that $|I_0|=\ell$. Then the image of $g$ lies in the subspace $H=\cap_{i\in I_0}H_i\cong\mathbb{CP}^{m-1-\ell}$. By the same arguments as in the proof of Proposition~\ref{generalized borel}, one deduces that  $g(\mathbb{C})$ lies in some hyperplane of $H$. Next, we follow the arguments as in  \cite[Example 3.10.21]{Kobayashi1998} (see also \cite{Green1975}). Suppose that $g_i$ satisfy the relation
\[
\sum_{k\in K} a_k g_k=0,
\]
where $K\subset J$, $|K|\geq 2$ and $a_k\in\mathbb{C}$ are nonzero constants. We claim that there exist two indexes $i,j\in k$ such that $g_i/g_j$ is constant. Indeed, if $|K|=2$, then we have nothing to prove. Otherwise, we consider the map $g_K=[g_k]_{[k\in K]}, \mathbb{C}\rightarrow H_k\cong\mathbb{CP}^{|K|-2}$, where $H_K$ is the hyperplane in $H$ defined by $\{\sum_{k\in K}z_k=0\}$, and the $|K|$ hyperplanes $\{z_k=0\} (k\in K)$. Using arguments as in the proof of Proposition~\ref{generalized borel}, one obtains further degeneracy. Inductively, the claim is proved.

Let $\sim$ be the equivalence relation on the index set $J$ defined as $i\sim j$ if and only if $g_i/g_j$ is constant and let $\{I_1,\dots, I_{\ell}\}$ be the partition of $J$ by $\sim$. Then for each $1\leq s\leq \ell$, we pick an index $i_s\in I_s$ and put
$g_j=\ell_jg_{i_s}$. Set $b_s=\sum_{j\in I_s}\ell_j$, then
\[
\sum_{i=0}^{n}g_i=\sum_{s=1}^{\ell}b_sg_{i_s}.
\]
It suffices to prove that the set $M=\{s:1\leq s\leq\ell: b_s\not=0\}$ consists only one elements. Indeed, if $|M|\geq 2$, then can use arguments as in Proposition~\ref{generalized borel} to get a nontrivial linear relation among $\{g_s\}_{s\in M}$. Then by the above claim, there exist two indexes $s_1,s_2\in M$ such that $g_{s_1}/g_{s_2}$ is constant, contradiction. Hence $M$ is of cardinality $1$. This finishes the proof of the Theorem.
\end{proof}

We can reformulate the above statement in a more geometric way as follows.
\begin{cor}
\label{degeneracy result from smt for D=generalized borel, logarithmic case}
Keeping the same assumption as in Proposition \ref{generalized borel}. Putting
\begin{align*}
\pi\colon\mathbb{CP}^n&\rightarrow\mathbb{CP}^n,\qquad [z_0:\dots:z_n]\mapsto [z_0^{d-\delta_0}Q_0:\dots:z_n^{d-\delta_n}Q_n],\\
g=[g_0:\dots:g_n]\colon \mathbb{C}&\rightarrow\mathbb{CP}^n,\qquad z\mapsto
\pi\circ f(z),
\end{align*}
where $g_i:=(z_i^{d-\delta_i}Q_i)\circ f$. If $f\colon\mathbb{C}\rightarrow\mathbb{CP}^n\setminus D$ is a non-constant holomorphic curve avoiding $D$, then the image of the composition map $g$ is contained in a smaller linear subspace of $\mathbb{CP}^n$ of dimension at most $\big[\frac{n}{2}\big]$.
\end{cor}

\begin{proof}
As before, we can find a partition $\{I_0,I_1\dots, I_{\ell}\}$  of $\{0,1,\dots,n\}$, all $I_i$ ($1\leq i\leq\ell$) have cardinality at least $2$ with at most one exception,  such that
\begin{enumerate}
	\item[(a)] $g_i\equiv 0$ for all $i\in I_0$;
	\item[(b)] For each $1\leq s\leq \ell$, pick an index $i_s\in I_s$, then for any $j\in I_s$ there exists a nonzero constant $\mu_{s,j}$  such that $g_{j}=\mu_{s,j} g_{i_s}$.
	\end{enumerate}
The image of $g$ lies in the linear subspace defined by the equations
\[
z_j=\mu_{s,j}z_{z_s}\quad (1\leq s\leq \ell,\quad j\in I_s,\, j\not=i_s);\qquad z_j=0 \quad (j\in I_0).
\]
Thus $\{g_i\}_{0\leq i\leq n}$ must satisfy at least $|I_0|+\sum_{s=1}^{\ell}( |I_s|-1)=n+1-\ell$ independent linear relations. Since $|I_s|\geq 2$ for at least $\ell-1$ indexes among $\{1,\dots,\ell \}$, there holds $\ell\leq \bigg[\dfrac{n+2}{2}\bigg]$, which implies that the image of $g$ is contained in a linear subspace of dimension at most $\bigg[\dfrac{n+2}{2}\bigg]-1=\bigg[\dfrac{n}{2}\bigg]$.
\end{proof}

Similarly, one can also get analog results in the compact case.

\begin{namedthm*}{Generalized Borel Theorem (compact case)}
\label{generalized borel theorem, compact case}
Let $d,n,\delta_0,\dots,\delta_n$ be integer numbers with  $\delta_i\geq 0$, $n\geq 2$ and $d>(n-1)(n+1)+\sum_{i=0}^{n}\delta_i$.  Suppose that the hypersurfaces  $D_i=\{z_i^{d-\delta_i}Q_i=0\}$ ($0\leq i\leq n+1$) have empty intersection. Then, for the collection of $n+1$ entire holomorphic functions $f_i$ ($0\leq i\leq n$) such that
\[
\sum_{i=0}^{n}f_i^{d-\delta_i}Q_i(f_0,f_1,\dots,f_n)\equiv 0,
\]
there is a partition of indexes $\{0,\dots,n\}=\cup_{\alpha=0}^{\ell}I_{\alpha}$ such that the followings hold
\begin{enumerate}
	\item[(i)] $f_i^{d-\delta_i}Q_i(f_0,f_1,\dots,f_n)\equiv 0$ if and only if $i\in I_0$ (of course the set $I_0$ may be empty);
	\item[(ii)] The cardinality $|I_{\alpha}|\geq 2$ for every $1\leq\alpha\leq \ell$;
	\item[(iii)] For each $1\leq \alpha\leq \ell$, for arbitrary indexes $i,j\in I_{\alpha}$, there exists a constant $c_{ij}\in\mathbb{C}$ such that $\frac{f_i^{d-\delta_i}Q_i(f_0,f_1,\dots,f_n)}{f_j^{d-\delta_j}Q_j(f_0,f_1,\dots,f_n)}=c_{ij}$;
	\item[(iv)] $\sum_{i\in I_{\alpha}}f_i^{d-\delta_i}Q_i(f_0,f_1,\dots,f_n)\equiv 0$ for all $0\leq\alpha\leq\ell$.
\end{enumerate}
\end{namedthm*}

\begin{cor}
\label{degeneracy result from smt for D=generalized borel, compact case}
Keeping the same assumption as above. Putting
\begin{align*}
\pi\colon\mathbb{CP}^n&\rightarrow\mathbb{CP}^n,\qquad [z_0:\dots:z_n]\mapsto [z_0^{d-\delta_0}Q_0:\dots:z_n^{d-\delta_n}Q_n],\\
g=[g_0:\dots:g_n]\colon \mathbb{C}&\rightarrow\mathbb{CP}^n,\qquad z\mapsto
\pi\circ f(z),
\end{align*}
where $g_i:=(z_i^{d-\delta_i}Q_i)\circ f$. If $f\colon\mathbb{C}\rightarrow D$ is a non-constant holomorphic curve into $D$, then the image of the composition map $g$ is contained in a smaller linear subspace of $\mathbb{CP}^n$ of dimension at most $\big[\frac{n-1}{2}\big]$.	
\end{cor}

\begin{rmk}
In the case where $\delta_0=\dots=\delta_n=0$, we recover the classical Borel Theorem \cite{Green1975}, \cite{Kobayashi1998}.
\end{rmk}

\section{Example of smooth hyperbolic surface whose complement is hyperbolically embedded in $\mathbb{CP}^3$}

\subsection{Genus of some plane algebraic curves}

We collect some results about the genus of plane algebraic curves (see \cite{Elgoul96} for explicit computations). By using the generalized Borel Theorem, these curves arise naturally as the degeneracy locus of entire holomorphic curves avoiding the Fermat type surface given in Theorem A. 
\begin{pro}
\label{genus computation Xd-2,2 Yd-2,2,betanot=0}
	For arbitrary nonzero constants $\beta,\epsilon_1,\epsilon_2\in\mathbb{C}\setminus\{0\}$, the plane curve $\mathcal{C}$ defined in inhomogeneous coordinates $(X,Y)$ of $\mathbb{CP}^2$ by the polynomial
	\[
	\beta + X^{d-2}(X^2+\epsilon_1^2)+ Y^{d-2}(y^2+\epsilon_2^2)
	\]
	has genus $g\geq 2$ when $d\geq 5$.
\end{pro}
\begin{pro}
\label{genus computation Xd-2,2 Yd-2,2}
	For arbitrary nonzero constants $\epsilon_1,\epsilon_2\in\mathbb{C}\setminus\{0\}$ with $(\pm i\epsilon_1)^d+(\pm i\epsilon_2)^d\not=0$, the plane curve $\mathcal{C}$ defined in inhomogeneous coordinates $(X,Y)$ of $\mathbb{CP}^2$ by the polynomial
	\[
	 X^{d-2}(X^2+\epsilon_1^2)+ Y^{d-2}(y^2+\epsilon_2^2)
	\]
	has genus $g\geq 2$ when $d\geq 4$.
\end{pro}

\subsection{Construction of smooth hyperbolic surface of low degree whose complement is hyperbolically embedded in $\mathbb{CP}^3$}
This section is devoted to prove Theorem A. Recalling that $D\subset\mathbb{CP}^3$ is the surface  defined by the homogeneous polynomial
\[
z_0^{d}+z_1^{d-2}(z_1^2+a_0z_0^2)+z_2^{d-2}(z_2^2+a_1z_0^2)+z_3^{d-2}(a_2z_1^2+a_3z_2^2+z_3^2).
\]
Using the generalized logarithmic Borel Theorem obtained in the previous section, we will prove the hyperbolicity of $\mathbb{CP}^3\setminus D$. The hyperbolicity of $D$ can be treated by the same arguments, using the compact version of the generalized Borel Theorem. Let $f\colon\mathbb{C}\rightarrow \mathbb{CP}^3\setminus D$ be a holomorphic function. We need to prove that $f$ is constant. As before we consider the map 
\[
\pi\colon\mathbb{CP}^3\rightarrow\mathbb{CP}^3, [z_0:z_1:z_2:z_3]\mapsto [z_0^d:z_1^{d-2}(z_2^2+a_0z_0^2):z_2^{d-2}(z_2^2+a_1z_0^2):z_3^{d-2}(a_2z_1^2+a_3z_2^2+z_3^2)].
\]

By the generalized Borel Theorem, the image of $f$ must satisfy two algebraic equations. Precisely, the image of $g:=\pi\circ f=[g_0:g_1:g_2:g_3]\colon \mathbb{C}\mapsto\mathbb{CP}^3$ must lie in a projective line $L$ which is the intersection of two planes defined by the linear polynomials of the form $\lambda_iz_i+\mu_j z_j=0$. To get further degeneracy, we need to study in very details all possibilities of the line $L$. We consider two cases:
\begin{enumerate}
\item[(a)] The image of $g$ lies in some coordinate plane;
\item[(b)] The image of $g$ does not lie in any coordinate plane.
\end{enumerate}

Let us consider the  case (a). If $g$ lands in another coordinate plane, then by Proposition~\ref{generalized borel}, $g$ must be constant and so is $f$. Thus we can suppose that the image of $g$ is contained in only one coordinate plane. We treat separately  the following possibilities for the position of the image of $g$:
\begin{enumerate}
\item[(a1)] The image of $g$ lies in the plane $z_0$;
\item[(a2)] The image of $g$ lies in the plane $z_i$ for $1\leq i\leq 2$;
\item[(a3)] The image of $g$ lies in the plane $z_3$.
\end{enumerate}

In the case (a1), one has $f_0\equiv 0$. Since $g(\mathbb{C})$ does not lie in any other coordinate plane, we need to consider three circumstances:
\begin{enumerate}
\item[(a1.i)] The image of $f$ lies in the plane curve $z_1^d+z_2^d=0$ and avoid $z_3^{d-2}(a_2z_1^2+a_3z_2^2+z_3^2)$;
\item[(a1.ii)] The image of $f$ lies in the plane curve  $z_1^d+z_3^{d-2}(a_2z_1^2+a_3z_2^2+z_3^2)=0$ and avoids $z_2=0$;
\item[(a1.ii')] The image of $f$ lies in the plane curve  $z_2^d+z_3^{d-2}(a_2z_1^2+a_3z_2^2+z_3^2)=0$ and avoids  $z_1=0$.
\end{enumerate} 

Consider the case (a1.i). The curve $f(\mathbb{C})$ lies in some projective line $z_1+\mu z_2=0$ where $\mu^d=-1$. This line intersects with $z_3^{d-2}(a_2z_1^2+a_3z_2^2+z_3^2)=0$ at three distinct points. This forces $f$ to be constant by the Little Picard Theorem.

Next we consider the case (a1.ii). The plane curve $z_1^d+z_3^{d-2}(a_2z_1^2+a_3z_2^2+z_3^2)=0$ in the inhomogeneous coordinate $(X,Y)$ is given by $X^d+a_2X^2+a_3Y^2+1=0$. This curve is smooth and therefore, by the genus formula, its genus is at least $2$, which forces $f$ to be constant. The case (a1ii') could be treated similarly.

Now we consider the case (a2), and without lost of generality, we assume that $g_1\equiv 0$. As before, we have three circumstances:
\begin{enumerate}
\item[(a2.i)] The image of $f$ lies in the surface  $z_0^d+z_2^{d-2}(z_2^2+a_1z_0^2)=0$ and avoids the surface $z_3^{d-2}(a_2z_1^2+a_3z_2^2+z_3^2)=0$;
\item[(a2.ii)]  The image of $f$ lies in the surface  $z_2^{d-2}(z_2^2+a_1z_0^2)+z_3^{d-2}(a_2z_1^2+a_3z_2^2+z_3^2)=0$ and avoids the surface $z_0=0$;
\item[(a2.iii)] The image of $f$ lies in the surface  $z_0^d+z_3^{d-2}(a_2z_1^2+a_3z_2^2+z_3^2)=0$ and avoids the surface $z_2^{d-2}(z_2^2+a_1z_0^2)=0$.
\end{enumerate} 

First we treat the case (a2.i). The image of $f$ lies in the planes $z_0=\lambda z_2$ and $z_1=0$ or $z_0=\mu z_1$ where $\lambda,\mu$ are some constants satisfying $\lambda^d+a_1\lambda^2+1=0$ and $\mu^2+a=0$. Hence its image lies in a line intersecting with $z_3^{d-2}(a_2z_1^2+a_3z_2^2+z_3^2)=0$ at three distinct points, which implies that $f$ is constant.

For the case (a2.ii), as before, the image of $f$ lies in the plane $z_1=0$ or $z_0=\mu z_1$. In the first case, the image of $f$ also lies in $z_2^{d-2}(z_2^2+a_1z_0^2)+z_3^{d-2}(a_3z_2^2+z_3^2)=0$. In inhomogeneous coordinates $(X,Y)$, this curve is given by the equation $X^{d-2}(X^2+a_1Y^2)+a_3X^2+1=0$. For all but except  finite choices of $a_1,a_3$, this curve is smooth and we can use the genus formula to conclude. In the second case where $f(\mathbb{C})$ lies in the plane $z_0=\mu z_1$, it must also lie in the curve $z_2^{d-2}(z_2^2+a_1z_0^2)+z_3^{d-2}(-a_0a_2z_0^2+a_3z_2^2+z_3^2)=0$. Again, for all but except finite choices of $a_i$, this curve is smooth and we get the constancy of $f$.

Finally, consider the case (a2iii). The image of $f$ lies in the plane $z_1=0$ or $z_0=\mu z_1$. Furthermore, it lies also in the plane curve $z_0^d+z_3^{d-2}(\alpha z_0^2+a_3z_2^2+z_3^3)=0$, where $\alpha\in\mathbb{C}$ is a constant. This curve is given in the inhomogeneous coordinate $(X,Y)$ as $X^d+\alpha X^2+a_3 Y^2+1=0$, which is of genus at least $2$, which implies that $f$ is constant.

We now move to the case (a3). As before, we have three sub-cases:
\begin{enumerate}
\item[(a3.i)] The image of $f$ lies in the surface  $z_0^d+z_1^{d-2}(z_1^2+a_0z_0^2)=0$ and avoids the surface $z_2^{d-2}(z_2^2+a_1z_0^2)=0$;
\item[(a3.i')]  The image of $f$ lies in the surface  $z_0^d+z_2^{d-2}(z_2^2+a_1z_0^2)=0$ and avoids the surface $z_1^{d-2}(z_1^2+a_0z_0^2)=0$;
\item[(a3.ii)] The image of $f$ lies in the surface  $z_1^{d-2}(z_1^2+a_0z_0^2)+z_2^{d-2}(z_2^2+a_1z_0^2)=0$ and avoids the surface $z_0=0$.
\end{enumerate} 

In the case (a3.i) and (a3.i'), the image of $f$ lies in a projective line or conic and avoids three points, and hence it must be constant. In the case (a3.ii), we employ Proposition~\ref{genus computation Xd-2,2 Yd-2,2} to conclude.

Now we pass to the case (b). By the generalized Borel Theorem, we need to consider the following circumstances:
\begin{enumerate}
	\item[(b1)] The image of $f$ lies in the surfaces $z_0^{d}+z_1^{d-2}(z_1^2+a_0z_0^2)=0$, $z_2^{d-2}(z_2^2+a_1z_0^2)+\alpha z_3^{d-2}(a_2z_1^2+a_3z_2^2+z_3^2)=0$ for $\alpha\not=1$ and avoids the surface $z_2^{d-2}(z_2^2+a_1z_0^2)=0$;
	\item[(b2)] The image of $f$ lies in the surfaces $z_0^{d}+\alpha z_1^{d-2}(z_1^2+a_0z_0^2)=0$, $z_2^{d-2}(z_2^2+a_1z_0^2)+ z_3^{d-2}(a_2z_1^2+a_3z_2^2+z_3^2)=0$ for $\alpha\not=1$ and avoids the surface $z_0^{d}+z_1^{d-2}(z_1^2+a_0z_0^2)=0$;
	\item[(b1')] The image of $f$ lies in the surfaces $z_0^{d}+z_2^{d-2}(z_2^2+a_1z_0^2)=0$, $z_1^{d-2}(z_1^2+a_0z_0^2)+\alpha z_3^{d-2}(a_2z_1^2+a_3z_2^2+z_3^2)=0$ for $\alpha\not=1$ and avoids the surface $z_1^{d-2}(z_1^2+a_0z_0^2)=0$;
   \item[(b2')] The image of $f$ lies in the surfaces $z_0^{d}+\alpha z_2^{d-2}(z_2^2+a_1z_0^2)=0$, $z_1^{d-2}(z_1^2+a_0z_0^2)+z_3^{d-2}(a_2z_1^2+a_3z_2^2+z_3^2)=0$ for $\alpha\not=1$ and avoids the surface $z_0^{d}+z_2^{d-2}(z_2^2+a_1z_0^2)=0$;
	\item[(b3)] The image of $f$ lies in the surfaces $z_0^d+  z_3^{d-2}(a_2z_1^2+a_3z_2^2+z_3^2)=0$, $z_1^{d-2}(z_1^2+a_0z_0^2)+\alpha z_2^{d-2}(z_2^2+a_1z_0^2)=0$ for $\alpha\not=1$ and avoids the surface $z_1^{d-2}(z_1^2+a_0z_0^2)=0$;
	\item[(b4)] The image of $f$ lies in the surfaces $z_0^d+  \alpha z_3^{d-2}(a_2z_1^2+a_3z_2^2+z_3^2)=0$, $z_1^{d-2}(z_1^2+a_0z_0^2)+ z_2^{d-2}(z_2^2+a_1z_0^2)=0$ for $\alpha\not=1$ and avoids the surface $z_3^{d-2}(a_2z_1^2+a_3z_2^2+z_3^2)=0$;
	\item[(b3')] The image of $f$ lies in the surfaces $z_0^d+ \alpha z_3^{d-2}(a_2z_1^2+a_3z_2^2+z_3^2)=0$, $z_1^{d-2}(z_1^2+a_0z_0^2)+ z_2^{d-2}(z_2^2+a_1z_0^2)=0$ for $\alpha\not=1$ and avoids the surface $z_3^{d-2}(a_2z_1^2+a_3z_2^2+z_3^2)=0$;
	\item[(b4')]The image of $f$ lies in the surfaces $z_0^d+  z_3^{d-2}(a_2z_1^2+a_3z_2^2+z_3^2)=0$, $z_1^{d-2}(z_1^2+a_0z_0^2)+\alpha z_2^{d-2}(z_2^2+a_1z_0^2)=0$ for $\alpha\not=1$ and avoids the surface $z_1^{d-2}(z_1^2+a_0z_0^2)=0$;
	\item[(b5)] The image of $f$ lies in the curve $z_0^{d}=\alpha z_1^{d-2}(z_1^2+a_0z_0^2)=\beta z_2^{d-2}(z_2^2+a_1z_0^2)=0$, where $\alpha,\beta \not=0, 1+\alpha^{-1}+\beta^{-1}=0$ and avoids $z_3^{d-2}(a_2z_1^2+a_3z_2^2+z_3^2)=0$;
	\item[(b6)] The image of $f$ lies in the curve $z_0^{d}=\alpha z_1^{d-2}(z_1^2+a_0z_0^2)=\beta z_3^{d-2}(a_2z_1^2+a_3z_2^2+z_3^2)$, where $\alpha,\beta \not=0, 1+\alpha^{-1}+\beta^{-1}=0$ and avoids $z_2^{d-2}(z_2^2+a_1z_0^2)=0$;
	\item[(b6')] The image of $f$ lies in the curve $z_0^{d}=\alpha z_2^{d-2}(z_2^2+a_1z_0^2)=\beta z_3^{d-2}(a_2z_1^2+a_3z_2^2+z_3^2)$, where $\alpha,\beta \not=0, 1+\alpha^{-1}+\beta^{-1}=0$ and avoids $z_1^{d-2}(z_1^2+a_0z_0^2)=0$;
	\item[(b7)] The image of $f$ lies in the curve $z_1^{d-2}(z_1^2+a_0z_0^2)=\alpha z_2^{d-2}(z_2^2+a_1z_0^2)=\beta z_3^{d-2}(a_2z_1^2+a_3z_2^2+z_3^2)$, where $\alpha,\beta \not=0, 1+\alpha^{-1}+\beta^{-1}=0$ and avoids $z_0=0$.
	\end{enumerate}
By symmetry, we only need to consider the case (bi). Details of case-by-case arguments are given belows.
\begin{enumerate}
\item[$\bullet$] In the case (b1),  the curve $f$ must lie in the plane  $z_1=\lambda z_0$, where $\lambda$ depending only on $a_0$. Hence $f(\mathbb{C})$ must lie in the plane curve $z_2^{d-2}(z_2^2+a_1z_0^2)+\alpha z_3^{d-2}(\mu z_0^2+a_3z_2^2+z_3^2)=0$, where $\mu=\lambda^2a_2$. In the inhomogeneous coordinates $(X,Y)$, this curve is given by the algebraic equation $X^{d-2}(X^2+a_1 Y^2)+\alpha(\mu Y^2+a_3X^2+1)=0$. A straightforwards computation shows that this curve has at most four ordinary singularities. Hence this curve is irreducible and has genus at least $2$, which yields the constancy of $f$.
\item[$\bullet$] The case (b2) can be treated similar as in the case (b1).
\item[$\bullet$] In the case (b3), applying Proposition~\ref{genus computation Xd-2,2 Yd-2,2}, the curve $z_1^{d-2}(z_1^2+a_0z_0^2)+\alpha z_2^{d-2}(z_2^2+a_1z_0^2)=0$ is either irreducible and has genus at least $2$, or it has irreducible components with lines of the form $z_1+\lambda z_3=0$ and an irreducible component of genus at least $2$. Therefore we can suppose that the image of $f$ lies in a line of the form $z_1+\lambda z_3=0$. This implies that $f(\mathbb{C})$ is contained in the plane curve $z_0^d+ z_3^{d-2}(\mu z_1^2+z_3^2)=0$. This curve in the inhomogeneous coordinates is given as $X^d+\mu Y^2+1=0$.
\item[$\bullet$] In the case (b4), we can apply directly Proposition~\ref{genus computation Xd-2,2 Yd-2,2} to get conclusion.
\item[$\bullet$] In the case (b5), the image of $f$ lies in a line and omits three distinct points there. The constancy of $f$ follows from the Little Picard Theorem.
\item[$\bullet$] In the case (b6), the image of $f$ lies in the plane $z_1=\lambda z_0$ and the algebraic curve $z_0^d=\beta z_3^{d-2}(\mu z_0^2+a_3z^2+z_3^2)=0$, or $X^d=\beta (\mu X^2+a_3 Y^2+a_3)=0$ in the inhomogeneous coordinates $(X,Y)$. This curve is irreducible and has at most two ordinary singularities, and hence its genus is at least $2$, which implies the constancy of $f$.
\item[$\bullet$] In the case (b7), the image of $f$ lies in a plane of the form $z_1=\lambda z_2$ for some constant $\lambda$. Therefore it lies also in the plane curve $z_1^{d-2}(z_1^2+a_0z_0^2)=\beta z_3^{d-2}(\mu z_1^2+z_3^2)$, or $X^{d-2}(X^2+a_0Y^2)-\beta (\mu X^2+1)=0$ in the inhomogeneous coordinates $(X,Y)$. Straightforward computation shows that this curve is irreducible and has genus at least $2$, where concludes the constancy of $f$.
\end{enumerate}
\begin{rem}
Similar constructions in the compact case were gave in \cite{Elgoul96}, \cite{siu_yeung1997}. The hyperbolic surfaces were constructed in these works are given by algebraic equations of the form
\[
z_0^{d-2}(z_0^2+\epsilon_1^2z_1^d+\epsilon_2^2z_2^2)+z_1^d+z_2^d+z_3^d=0,
\]
where $\epsilon_1,\epsilon_2$ are suitable constants.
Note however that the complement of these surfaces are never hyperbolic. Indeed, the complement of such surface contains the image of the holomorphic curve $f\colon\mathbb{C}\rightarrow\mathbb{CP}^3,z\mapsto [0:e^{\varphi}:\psi:\epsilon\psi]$, where $\varphi,\psi$ are entire functions and $\epsilon$ is a $d$--root of $-1$.
\end{rem}
\section{Holomorphic curves intersecting Fermat-Warning type hypersurface}
\subsection{Counting dimension of some subvarieties of Grassmanians}
Let $\Gr_{m,k}$ denote the Grassmannian of complex codimension $k$ subspaces of $\mathbb{C}^m$. Let $Q_{m,k}$ denote the subspace $0_k\times\mathbb{C}^{m-k}\subset\Gr_{m,k}$. For integers $a,b,c$ with $1\leq a\leq c\leq m-1$ and $1\leq b \leq c\leq a+b$, consider the set
$$
\Gamma_{m,a,b,c}:=\{V\in \Gr_{m,a}:\dim V \cap Q_{m,c}\geq m-c\}.
$$
The following counting dimension result is due to Shiffman-Zaidenberg \cite{shiffman_zaidenberg2002_pn}.
\begin{lem}
\label{counting dimension grassmanian}
	\[
	\dim \Gr_{m,a}-\dim\Gamma_{m,a,b,c}=(m-c)(a+b-c).
	\]	
\end{lem} 
\subsection{Second Main Theorem for nonconstant holomorphic curves into projective and a Fermat-Warning type hypersurface}
Let us now enter the proof of Theorem B. Let $[z_0:\dots:z_n]$, $[w_1:\dots:w_m]$ be homogeneous coordinates of $\mathbb{CP}^n$ and $\mathbb{CP}^{m-1}$, respectively. Let $[f_0:f_1:\dots:f_n]$ be a reduced representation of $f$. Consider the morphism $$\pi\colon\mathbb{C}\mathbb{P}^n\longmapsto\mathbb{C}\mathbb{P}^{m-1},\qquad [z_0:z_1:\dots:z_n]\longmapsto [h_1^d:h_2^d:\dots:h_m^d],
$$
and put $g=\pi\circ f\colon\mathbb{C}\longmapsto \mathbb{P}^{m-1}, z\longmapsto [g_1^d:\dots:g_m^d]$, where $g_i:=h_i\circ f$. It follows from the definition that
\begin{equation}
\label{order function of f and g}
T_g(r)=d\,T_f(r)+O(1).
\end{equation}
Let $\{H_i\}_{1\leq i\leq m+1}$ be the family of $m+1$ hyperplanes in general position in $\mathbb{CP}^{m-1}$ given by
\begin{align*}
H_i&=\{w_i=0\}&(1\leq i\leq m),\\
H_{m+1}&=\{\sum_{j=1}^{m}w_j=0\}.&\quad
\end{align*}

Set $I_0:=\{i\in\mathbb{N}: 1\leq i\leq m,g_i\equiv0\}$ and assume that $\ell=|I_0|$. Then $0\leq \ell\leq n-1$ and the image of $g$ lies in the subspace $H=\cap_{i\in I_0}H_i\cong\mathbb{CP}^{m-1-\ell}$. Denote by $J=\{1\leq j\leq m\}\setminus I_0$, and consider the map $$
\widetilde{g}\colon \mathbb{C}\mapsto H\cong \mathbb{CP}^{m-1-\ell},\qquad z\rightarrow [g_j^d(z)]_{j\in J},
$$
we then still have
\begin{equation}
\label{order function of widetilde f and g}
T_{\widetilde{g}}(r)=d\,T_f(r)+O(1).
\end{equation}

If $\widetilde{g}$ is linearly nondegenerate, then by  applying Cartan's Second Main Theorem for the collection of hyperplanes $\{\widetilde{H_j}\}_{j\in J\cup\{m+1\}}$,
where $\widetilde{H_j}=H_j\cap H$, one obtains
\begin{equation}
\label{using cartan smt for g}
T_{\widetilde{g}}(r)
\leq
\sum_{j\in J}N_{\widetilde{g}}^{[m-1-\ell]}(r,\widetilde{H_j})
+
N_{\widetilde{g}}^{[m-1-\ell]}(r,\widetilde{H_{m+1}})
+
S_{\widetilde{g}}(r).
\end{equation}
For each $j\in J$, since the multiplicity of $\widetilde{g}^*\widetilde{H_j}$ is at least $d$ at every point on its support, one has
\[
N_{\widetilde{g}}^{[m-1-\ell]}(r,\widetilde{H_j})\leq\bigg(\dfrac{m-1-\ell}{d}\bigg)N_{\widetilde{g}}(r,\widetilde{H_i}),
\]
which implies
\[
N_{\widetilde{g}}^{[m-1-\ell]}(r,\widetilde{H_j})\leq\bigg(\dfrac{m-1-\ell}{d}\bigg)T_{\widetilde{g}}(r)+O(1),
\]
by the First Main Theorem. Furthermore,  it is clear from the definition that 
$$
N_{\widetilde{g}}^{[m-1-\ell]}(r,\widetilde{H_{m+1}})
=
N_f^{[m-1-\ell]}(r,D).
$$
Thus it follows from \eqref{order function of widetilde f and g}, \eqref{using cartan smt for g} that
\begin{equation}
\label{smt after I0}
\big(d-(m-\ell)(m-\ell-1)\big)T_f(r)
\leq
N_f^{[m-1-\ell]}(r,D)+S_f(r),
\end{equation}
which implies the desired estimate of the Main Theorem. 

Next, consider the case where the image of $g$ lies in some hyperplane  of the subspace $H\cong\mathbb{CP}^{m-1-\ell}$. Suppose that the holomorphic functions $g_j,j\in J$  satisfy the equation
\begin{equation}
\label{relation in K}
\sum_{k\in K}a_kg_k^d=0,
\end{equation}
where $K\subset J$ is some subset of $J$ with $|K|\geq 2$ and where $a_k$ are nonzero complex numbers.
We first prove the following
\begin{claim}
	There exist some indexes $i\not=j$ with $i,j\in J$ such that $g_i/g_j$ is constant.
\end{claim}

Indeed, if $|K|=2$, the claim follows directly from \eqref{relation in K}. Suppose $|K|=\gamma\geq3$. Consider the homogeneous coordinates $[w_k]_{k\in K}$ of the linear subspace $\mathbb{C}\mathbb{P}^{\gamma-1}$. Similar as in above, consider the hyperplane $H_K\cong \mathbb{C}\mathbb{P}^{\gamma-2}\subset \mathbb{C}\mathbb{P}^{\gamma-1}$ defined as $\sum_{k\in I}a_kw_k=0$ and the holomorphic map 
$$
\mathsf{g}_{K}\colon\mathbb{C}
\mapsto
H_K\cong \mathbb{C}\mathbb{P}^{\gamma-2},\qquad \mathsf{g}_K(z)=[g_k^d(z)]_{k\in K}.
$$

If the image of $\mathsf{g}_{K}$ doesn't lie in some smaller linear subspace, then using ramification theorem for $\mathsf{g}_{K}$ and the family of $\gamma$ hyperplanes $\{H_k\}_{k\in K}$ in $H_K\cong \mathbb{C}\mathbb{P}^{\gamma-2}$, one obtains 
\begin{equation}
\label{applying ramification for gI}
\sum_{k\in K}\bigg(1-\dfrac{\gamma-2}{d}\bigg)\leq \gamma-1,
\end{equation}
contradiction. Thus $\mathsf{g}_{K}$ is linearly degenerate. Inductively, the claim is proved.

Going back to the proof of the Theorem, we now follow the arguments in \cite[Example 3.10.21]{Kobayashi1998}. Let $\sim$ be the equivalence relation on the index set $J$ defined as $i\sim j$ if and only if $g_i/g_j$ is constant and let $\mathcal{I}:=\{I_1,\dots, I_r\}$ be the partition of $J$ by $\sim$. For each $1\leq s\leq r$, suppose that $|I_s|=\kappa_s$. For convenient, we put $|I_0|=\ell=\kappa_s$ and we write $I_s=\{i_{s,1},\dots,i_{s,\kappa_s}\}$. Then 
\[
g_{i_{s,j}}=\mu_{s,j}g_{i_{s,1}}\eqno (\forall\, 1\,\leq\,s\,\leq\,r,\,\forall\, 2\,\leq\,j\,\leq\,\kappa_s),
\]
for some constants $\mu_{s,j}\in\mathbb{C}$.
we pick an index $i_s\in I_s$ and put
$g_j=\ell_jg_{i_s}$ for each $j\in I_s$. Set $b_s=1+\sum_{j=2}^{\kappa_s}\mu_{s,j}^d (1\leq s\leq r)$, then
\[
\sum_{j\in J}g_j^d=\sum_{s=1}^rb_sg_{i_{s,1}}^d.
\]

Now, put $M=\{s:1\leq s\leq r, b_s\not=0\}$. Consider the case where $|M|\geq n+1$. Similar as in above, one considers the map $$
\mathsf{g}_{M}\colon\mathbb{C}
\mapsto
H_M\cong \mathbb{C}\mathbb{P}^{|M|-1},\qquad \mathsf{g}_M(z)=[g_{i_{s,1}}^d(z)]_{s\in M}.
$$
Since $|M|\geq n+1$, one has $T_{\mathsf{g}_{M}}(r)=d\,T_f(r)$. If $\mathsf{g}_{M}$ is linearly degenerate, by the above claim, one has $g_i/g_j$ is constant for some $i,j\in M$, a contradiction. Hence $\mathsf{g}_{M}$ is linearly nondegenerate, and
by the same arguments as in \eqref{smt after I0}, one can use Cartan'Second Main Theorem for $\mathsf{g}_{M}$ and for the $|M|+1$ hyperplanes $\{w_s=0\} (s\in M), \{\sum_{s\in M}b_s w_s=0\}$, and one gets the desired estimate.

Now, consider the case where $|M|\leq n$. For each $s\notin M$, the set $I_s$ contains at least $2$ indexes. The image of $g$ lies in the $r$-plane $\mathcal{Y}_{\mu}^{\mathcal{I}}$ given by equations
\[
w_{i_{s,j}}=\mu_{s,j}w_{i_{s,1}}\quad (1\,\leq\,s\,\leq\,r,\, 2\,\leq\,j\,\leq\,\kappa_s); \qquad w_{i_{0,j}}=0\quad (1\,\leq\,j\,\leq\,\kappa_0),
\]
where $\mathcal{I}=\{I_1,\dots,I_r\}, \mu=\mu_{s,j}$. Let $\Gr_{m,k}$ denote the Grassmannian of complex codimension $k$ linear subspace of $\mathbb{C}^m$. We need to check that for generic $V\in \Gr_{m,m-n-1}$, the intersection $\mathcal{Y}_{\mu}^{\mathcal{I}}\cap \mathbb{P}(V)$ is at most a point for all such above $(\mathcal{I},\mu)$. Indeed, by Lemma~\ref{counting dimension grassmanian}, $\mathcal{Y}_{\mu}^{\mathcal{I}}\cap \mathbb{P}(V)$ is either a point or empty, unless $V$ lies in a subvariety of $\Gr_{m,m-n-1}$ of codimension
\[
\alpha=2(m-n-r+1).
\]

On the other hand, since $b_s=1+\sum_{j=2}^{\kappa_s}\mu_{s,j}^d=0$ for any $s\not\in M$, the $\mu$-moduli space of the lifting $\widetilde{\mathcal{Y}}_{\mu}^{\mathcal{I}}$ of $\mathcal{Y}_{\mu}^{\mathcal{I}}$ to $\mathbb{C}^m$ in $\Gr_{m,m-n-1}$ has dimension
\[
\beta\leq
\sum_{s\in M}(\kappa_s-1)
+
\sum_{1\,\leq\, s\leq\, r,\,s\,\not\in\, M}(\kappa_s-2)
=\sum_{s=1}^{r}(\kappa_s-2)+|M|
=m-\kappa_0-2r+|M|
\leq
m-2r+n.
\]

Since $m\geq 3n-1$, one has $\alpha>\beta$. Therefore, for generic $V\in \Gr_{m,m-n-1}$, the intersection $\mathcal{Y}_{\mu}^{\mathcal{I}}\cap \mathbb{P}(V)$ is  a point or empty for all $(\mathcal{I},\mu)$. Hence $g$ is constant and so is $f$, a contradiction. This finishes the proof of Theorem B.
\begin{rem}
	To get the hyperbolicity result of the complement, one only needs the condition $m\geq 2n$. However, to get Second Main Theorem estimate, in the above proof, we need the comparability between $T_{\mathsf{g}_M}(r)$ and $T_f(r)$. For this reason, we must put stronger assumption $m\geq 3n-1$.
\end{rem}

\begin{center}
	\bibliographystyle{plain}
	\bibliography{references}
\end{center}
\address
\end{document}